\documentclass[submission,copyright,creativecommons]{eptcs}
\providecommand{\event}{AiML 2026} 

\usepackage{aiml26}

\usepackage{iftex}
\usepackage{newtxtext,newtxmath}

\usepackage{enumitem}

\ifpdf
  \usepackage{underscore}         
  \usepackage[T1]{fontenc}        
\else
  \usepackage{breakurl}           
\fi

\usepackage{amsmath}

\usepackage{stmaryrd}
\usepackage{tikz-cd}
  \usetikzlibrary{arrows,arrows.meta}
  \tikzcdset{arrow style=tikz, diagrams={>=stealth'}}

\usepackage{multicol}
\usepackage{proof}
\usepackage{booktabs}
\usepackage{wrapfig}

\newcommand{\mc}[1]{\mathcal{#1}}

\newcommand{\mf}[1]{\mathfrak{#1}}

\newcommand{\mo}{\mathfrak}

\newcommand{\lna}[1]{\ensuremath{\mathsf{#1}}}

\newcommand{\flatAx}{\mathrm{Ax^{\flat}}}

\newcommand{\Th}{\mathrm{Th}}
\newcommand{\sneg}{\mathord{\sim}}

\newcommand{\fleq}{\preceq}    
\newcommand{\fgeq}{\succeq}    

\newcommand{\Vdashs}{\Vdash^s}  

\DeclareSymbolFont{symbolsC}{U}{txsyc}{m}{n}
\DeclareMathSymbol{\sto}{\mathrel}{symbolsC}{74}

\renewcommand{\iff}{\quad\text{iff}\quad}
\renewcommand{\phi}{\varphi}
\DeclareMathOperator{\Prop}{Prop}

\newcommand{\Ax}{\mathrm{Ax}}

\newcommand{\llb}{\llbracket}
\newcommand{\rrb}{\rrbracket}
\newcommand{\seq}{\Rightarrow}

\newcommand{\hlcf}{\lna{HLC^{\flat}}}
\newcommand{\hlcs}{\lna{HLC^{\sharp}}}

\DeclareMathOperator{\up}{up}
\newcommand{\Lsto}{\mathcal{L}_{\sto}}

\newcommand{\SEG}{\mathrm{SEG}}
\newcommand{\subsetsim}{\mathrel{\substack{\textstyle\subset\\[-0.2ex]\textstyle\sim}}}

\newcommand{\axKa}{\mathsf{k_a}}
\newcommand{\axTr}{\mathsf{tr}}
\newcommand{\axDi}{\mathsf{di}}
\newcommand{\axEm}{\mathsf{em}}

\newcommand{\axTb}{\mathsf{t_{\Box}}}
\newcommand{\strength}{\mathsf{str}}
\newcommand{\foura}{\mathsf{4_a}}

\newcommand{\pa}{\mathsf{p_a}}

\newcommand{\tto}{\sto}

\newcommand{\myitem}[1]{
  \renewcommand{\labelenumi}{(\theenumi) }
  \renewcommand{\theenumi}{#1}
  \item
}

\newcommand{\undto}{\mathrel{\mkern1mu\underline{\mkern-1mu \to\mkern-2mu}\mkern2mu }}

\newcommand{\undstof}{\mathrel{\mkern1mu\underline{\mkern-1mu \sto\mkern-1mu}\mkern1mu }^\flat}

\newcommand{\duf}[1]{\mf{#1}^{+\flat}}

\title{Relational Semantics for Flat Heyting-Lewis Logic}
\author{%
  Jim de Groot\footnote{The first author was supported by Swiss National Science Foundation (SNSF) grant No.~200021\_215157.}
  \institute{University of Bern\\ Bern, Switzerland}
  \email{jim.degroot@unibe.ch}
\and
  Tadeusz Litak\footnote{The second author would like to acknowledge the support of  PNRR MUR projects PE0000013-FAIR and F53C25001420001-GAMEL.}
  \institute{University of Naples Federico II\\ Naples, Italy}
  \email{tadeusz.litak@fau.de} \email{tadeusz.litak@unina.it}
}

\newcommand{\titlerunning}{Relational Semantics for $\hlcf$}
\newcommand{\authorrunning}{J.~de Groot \& T.~Litak}

\hypersetup{
  bookmarksnumbered,
  pdftitle    = {\titlerunning},
  pdfauthor   = {\authorrunning},
  pdfsubject  = {EPTCS},               
  pdfkeywords = {keyword1, keyword2} 
}

\newcommand{\deq}{:=}
\newcommand{\la}{\langle}
\newcommand{\ra}{\rangle}

\newcommand{\stex}[2]{#2^{\lceil #1 \rceil}}

\newcommand{\gA}{\mathcal{A}}

\newcommand{\opr}{\Box}

\newcommand{\ipc}{\lna{IPC}}

\newcommand{\iam}{\hlcf}

\newcommand{\hyph}{\mbox{-}}

\newcommand{\hae}{\ensuremath{\textsc{hae}}}
\newcommand{\hao}{\ensuremath{\textsc{hao}}}

\newcommand{\hlalg}{\lna{L}\hyph\hae}
\newcommand{\dhlalg}{\lna{L}\hyph\hao}
\newcommand{\thaes}{\ensuremath{\hlalg\textup{s}}}

\newcommand{\takeout}[1]{}

\begin{document}
\maketitle

\begin{abstract}
  We introduce relational semantics for
  ``flat Heyting-Lewis logic'' $\mathsf{HLC}^{\flat}$.
  This logic arises as the extension of intuitionistic logic
  with a Lewis-style strict implication modality that,
  contrary to its ``sharp'' counterpart $\mathsf{HLC}^{\sharp}$,
  does not turn meets into joins in its first argument.
  We prove completeness and the finite model property for $\mathsf{HLC}^{\flat}$
  and for several extensions with additional axioms.
\end{abstract}

\section{Introduction}

  Recent years have seen a revival of the interest in intuitionistic modal logics~\cite{proietti2012,litak14:trends,stepaiale15,artemovp16:rsl,litakpr17,BalbianiBDF20,rog20b,Kavvos20,DasM23,ShillitoGGI23,GirlandoEA23,AlmeidaB24,BalbianiEA24,DasEA24,groshiclo25,BozzelliCCM25,Mojtahedi26,BalbianiG26},
  including extensions with a Lewisian strict implication $\tto$~\cite{LitVis18,litvis24}
  and various types of conditional implications~\cite{wei19,DalGir26}.
Recall that in the intuitionistic setting, $\tto$ is not definable in terms of unary $\Box$. Instead, it can be viewed as sitting between $\Box(p \to q)$ and $\Box p \to \Box q$.
Indeed, the basic ``flat'' system\footnote{Naming underwent several evolutions. Early references in the Utrecht school \cite{iemh:prov01,iemh:pres03,Zhou03,IemhoffJZ05:igpl} denoted the base ``flat'' system as \lna{iP^-} and the base sharp system as \lna{iP}, Litak and Visser \cite{LitVis18} replaced \lna{iP} with \lna{iA} (with $\lna{P}$ standing for \emph{preservativity} and $\lna{A}$ standing for \emph{arrows}), and in a subsequent paper the same authors \cite{litvis24} finally settled for the present notation.} $\hlcf$ proves~\cite{LitVis18}:
  \begin{enumerate}\itemindent=1.1em
	\myitem{$\lna{bl}$} $\Box(p \to q) \to (p \tto q)$
	\myitem{$\lna{lb}$}\label{ax:lb} $(p \tto q) \to (\Box p \to \Box q)$.
  \end{enumerate}
The original motivation of the Utrecht school to study such a connective came from research on schematic logics of theories over intuitionistic arithmetic \lna{HA}. More specifically, it arose from the study of {\emph{$\Sigma^0_1$-preservativity}} \cite{viss:eval85,viss:prop94,viss:subs02,iemh:pres03,IemhoffJZ05:igpl}, which over \lna{PA} can be seen as the contraposed variant of both \emph{$\Pi^0_1$-conservativity} and \emph{arithmetical interpretability}   \cite{bera:inte90,shav:rela88,japa:logi98,viss:over98,arte:prov04}.
Subsequently, many other application and interpretations were put forward, see e.g.~\cite{LitVis18,grolitpat26-arxiv}.
In particular, in the presence of an additional axiom $\strength$ (see Section \ref{sec:extensions}) the resulting calculus turns out to be the Curry-Howard counterpart (i.e.~the inhabitation logic) of  \emph{(Hughes) arrows} in functional programming, in particular in Haskell~\cite{Hug00,Hug04,atkey08:msfp,jacobshh09:jfp, lindleywy08:msfp,lindleywy10:jfp,lindley14:wgp}. Somewhat underdeveloped philosophical applications include a generalisation of intuitionistic epistemic logic \lna{IEL}~\cite{artemovp16:rsl} to the \emph{intuitionistic logic of entailments} \lna{IELE}~\cite[Section~2.4]{grolitpat26-arxiv} or a fine-grained analysis of the collapse of Lewis' original 1918 system of strict implication caused by involutive negation~\cite[Appendix~D]{LitVis18}.\footnote{It is worth noting here that in later years, having become aware of nascent study of non-classical calculi, Lewis not only followed closely the development of early multi-valued logics, but also on at least one occasion spoke favourably of Brouwer's rejection of excluded middle. More information and detailed discussion can be found in Litak and Visser \cite{LitVis18}.}

  The flat calculus~$\hlcf$ arises from extending intuitionistic logic with
  a binary operator $\sto$ that is normal in it second argument, transitive,
  and satisfies \emph{implication necessitation}, i.e.~derivability of $\phi \to \psi$
  implies derivability of $\phi \sto \psi$.
  From this, we can obtain the sharp calculus $\hlcs$ by adding the axiom:
  \begin{enumerate}\itemindent=1.1em
    \myitem{\lna{di}}\label{ax:di} $(p \sto r) \wedge (q \sto r) \to ((p \vee q) \sto r)$
  \end{enumerate}
  This sharp version of the logic can conveniently be interpreted in
  Kripke-style relational semantics. This perspective has resulted in numerous
  correspondence, completeness and finite model property results for this ``sharp'' semantics \cite{iemh:prov01,iemh:pres03,Zhou03,IemhoffJZ05:igpl,LitVis18},
  with recent work showing how to use the natural G{\"o}del-McKinsey-Tarski
  translation to transfer metatheory of bimodal classical logics~\cite{grolitpat26-arxiv}.

Nevertheless, from the point of view of applications and interpretations discussed above, there are few reasons to insists on the ``sharp'' calculus $\hlcs$ as the minimal one, excluding interpretations not validating \lna{di}. In more philosophical contexts such as \lna{IELE}~\cite[Section~2.4]{grolitpat26-arxiv}, the present state of the research does not determine the status of \lna{di}, whereas in mathematical and computer science applications enforcing it as a base axiom would simply be too restrictive.
Insisting on the validity of \lna{di} in the inhabitation logic of arrows in functional programming languages~\cite{Hug00,Hug04,atkey08:msfp,jacobshh09:jfp, lindleywy08:msfp,lindleywy10:jfp,lindley14:wgp} would limit the Curry-Howard interpretation to so-called \emph{arrows with choice} \cite{Hug04}. While this would still cover \emph{arrows with delay} and  \emph{higher-order arrows} (corresponding, respectively, to \emph{applicative functors} and monads), and some other instances such as list processors and Kleisli arrows, important examples of arrows \emph{without} choice can be obtained using automata or functions on infinite streams \cite[\S\ 2.3]{grolitpat26-arxiv}. 

The situation is similar when it comes to schematic logics of arithmetical theories, i.e.~the original motivation of the Utrecht school to study variants and extensions of  \lna{HLC}: Sharpness obtains in some contexts, but not in others. For example, it fails classically, i.e.~when \lna{HA} is replaced by Peano Arithmetic (\lna{PA}). In fact, one can argue that the main reason why classical interpretability logic is a separate subfield with its own methods and semantics, irreducible even to multi-modal and multi-dimensional normal modal logics, is precisely the failure of the validity of the (contraposed form of) \lna{di} in the logic of $\Pi^0_1$-conservativity/interpretability of \lna{PA}. Some readers might find this failure surprising
given that \lna{di} holds in the schematic theory of \lna{HA}, which is a subtheory of \lna{PA}. However, non-monotonicity of schematic logics is a phenomenon observable even in the signature of provability with a single unary $\Box$ (\cite[\S\ 5.3]{LitVis18}). A sufficient condition for sharpness of the schematic logic of a given arithmetical theory $T$ is that $T$ is able to prove that the set of $T$-theorems is \emph{closed under q-realisability} \cite[Theorem~5.6]{LitVis18}. Note that if $T' \supseteq T$, the $T$-provable statements about $T$-theorems remain $T'$-provable, but there might be new $T'$-theorems for which q-realisability is not $T'$-provable (and, \emph{a fortiori}, not $T$-provable).

Arguably, the main temptation to accept the sharp system as the minimal one has been that of nice completeness results and natural countermodels. In the flat setting, so far one has had to turn to algebraic semantics
  or a suitable adaptation of
  Chellas-Weiss semantics for $\lna{ICK}$ \cite{wei19,cialiu19,dufgro25},
  Routley-Meyer semantics for substructural logics~\cite{roumey72a,roumey72b,roumey73,restall00,bimdunfer18},
  or (generalised) Veltman semantics \cite{dejo:prov90,verb:unpu92,dejo:comp99,joosten2020},
  because a simple Kripke-style semantics for $\hlcf$ appeared elusive. 

  In this paper we fill this gap by providing a Kripkean interpretation for $\hlcf$.
  This semantics is inspired by recent work on semantics of \lna{CK}~\cite{groshiclo25},
  and crucially relies on using a \emph{preorder} $\fleq$
  instead of a partial order to interpret the intuitionistic implication.
  Since the semantic clause for $\tto$ directly enforces upward persistence
  (Definition~\ref{def:cfrm}), the most general version of the new semantics
  (Definition~\ref{def:cfrm}) does not impose any interaction conditions between
  $R$ and $\fleq$.
  However, similarly to the case of intuitionistic $\Box$ and unlike the sharp interpretation,
  our language is oblivious to closing $R$ under \emph{post-composing} with $\fleq$
  (Proposition \ref{prop:upsame}), and the resulting  \emph{upward-flat} frames (Section \ref{sec:upflat}) prove convenient for computing correspondents and obtaining completeness results.
  
  Using a canonical model construction, we prove completeness and the finite
  model property for $\hlcf$ and several of its extensions
  (Sections~\ref{sec:canonical} and~\ref{sec:reuse}).
  Guided by the canonical model construction for $\lna{CK}$, we use
  \emph{segments} rather than prime theories to have a more fine-grained
  handle on the modal accessibility relation.
  Still mirroring $\lna{CK}$, we sometimes need to restrict our choice of
  segments, for example when proving completeness for natural variants
  of \lna{K4} and \lna{S4} in our setting (Section \ref{subsec:alt-can-mod}). 

When $\preceq$ is collapsed to equality, turning our frames into standard Kripke frames, our frames turn~\lna{lb} into a bi-implication, rather than \lna{bl}. 
This does not mean that our semantics trivialises classically: in the preorder setting, validating excluded middle simply requires $\fleq$ to be symmetric, and such a classical variant of our semantics does not collapse $\tto$ (Example~\ref{ex:classical}). This creates the opportunity to use our semantics for completeness results for subsystems of standard interpretability logics such as \lna{ILM} and~\lna{ILP}.

In the \lna{CK} setting, the segment approach can be used to obtain duality results~\cite{groshiclo26}. While our paper does not discuss duality in depth, we include comments for the interested reader, such as Remarks~\ref{rem:dualityhard} and~\ref{rem:nogmt}, illustrating difficulties with more standard approaches. 
However, we discuss a promising application in Section \ref{sec:exsta} in the context of syntactically motivated notion of \emph{extension stability}. We note the relationship of this notion to what one might call the \emph{open subframe} construction, and use our semantics to show that $\hlcs$ is not extension stable, unlike the flat base calculus.

\section{Intuitionistic strict implication, sharply and flatly} \label{sec:base}

  This section provides preliminaries and recapitulates known material.
  Section~\ref{sec:baseflat} presents the base flat system \hlcf.
  Section~\ref{sec:sharpening} discusses the sharp variant \hlcs\ together with
  its known Kripke semantics. Section~\ref{sec:alg} recapitulates the algebraic
  semantics of both systems. Throughout the paper, we denote by $\Lsto$ the
  language generated by the grammar
  \begin{equation*}
    \phi ::= p \mid \top \mid \bot \mid \phi \wedge \phi \mid \phi \vee \phi \mid \phi \to \phi \mid \phi \sto \phi,
  \end{equation*}
  where $p$ ranges over some arbitrary but fixed set $\Prop$ of proposition letters.
  We abbreviate $\Box\phi := \top \sto \phi$.

\subsection{Syntax and axioms of the base flat system}\label{sec:baseflat}

  A \emph{consecution} is an expressions of the form $\Gamma \seq \phi$,
  where $\Gamma \cup \{ \phi \} \subseteq \Lsto$.

\begin{definition}
  Let $\flatAx$ be an axiomatisation of intuitionistic logic
  together with the axioms
  \begin{multicols}{2}%
  \begin{enumerate}\itemindent=1.1em
	\myitem{$\axKa$} $((p \sto q) \land (p \sto r)) \to (p \sto (q \land r))$
	\myitem{$\axTr$} $((p \sto q) \wedge (q \sto r)) \to (p \sto r)$
  \end{enumerate}
  \end{multicols}
  \noindent
  If $\Ax \subseteq \Lsto$, the we denote by $\mathcal{I}(\Ax)$ the
  collection of substitution instances of formulas in $\Ax$,
  and define the axiomatic system $\hlcf \oplus \Ax$ by:
  \begin{equation*}
    \mathsf{(Ax)} \;
      \dfrac{\phi\in\mathcal{I}(\flatAx)\cup \mathcal I (\Ax)}{\Gamma \seq \phi} \qquad
    \mathsf{(El)} \;
      \dfrac{\phi \in \Gamma}{\Gamma \seq \phi} \qquad
    \mathsf{(MP)} \; 
      \dfrac{\Gamma \seq \phi \qquad \Gamma \seq \phi \to \psi}{\Gamma \seq \psi} \qquad
    \mathsf{(N_a)} \;
      \dfrac{\emptyset \seq \phi \to \psi}{\Gamma \seq \phi \sto \psi}
  \end{equation*}
  We say that $\Gamma \seq \phi$ is \emph{provable in} $\hlcf \oplus \Ax$,
  and write $\Gamma \vdash_{\Ax} \phi$, if
  there exists a tree of consecutions built using the rules above with
  $\Gamma \seq \phi$ as root and adequate applications of rules
  $\mathsf{(El)}$ and $\mathsf{(Ax)}$ as leaves.
  If $\Ax = \emptyset$ then we abbreviate $\vdash_{\Ax}$ to $\vdash$,
  and if $\Ax=\{ \phi_1, \dots, \phi_n\}$ then we write
  $\hlcf \oplus \phi_1 \oplus \dots \oplus \phi_n$ for $\hlcf \oplus \Ax$.
  Finally, we write $\Gamma \vdash_{\Ax} \Delta$ if there exist $\psi_1, \ldots, \psi_m \in \Delta$ such that
  $\Gamma \vdash_{\Ax} \psi_1 \vee \cdots \vee \psi_m$.
\end{definition}

\begin{proposition}
  For $\Ax \subseteq \Lsto$ and uniform substitution $\sigma$,
  the following rules are admissible in $\hlcf \oplus \Ax$:
  \begin{equation*}
    \infer{\Gamma, \Gamma' \seq \phi}
          {\Gamma \seq \phi}
    \qquad
    \infer{\Gamma \seq \phi}
          {\{\Gamma \seq \delta \; \mid \; \delta \in \Delta \} \quad \Delta \seq \phi}
    \qquad
	\infer{\Gamma^\sigma \seq \varphi^\sigma}
	      {\Gamma \seq \varphi}
	\qquad
    \infer={\Gamma \seq \phi \rightarrow \psi}
          {\Gamma, \phi \seq \psi}
  \end{equation*}
\end{proposition}
\begin{proof}
  By induction on the height of a derivation for the premiss(es).
\end{proof}

  The first three rules show that $\hlcf \oplus \Ax$ is a monotone,
  compositional and structural relation, respectively.
  The double line in the last rule indicates that it holds both ways.
  Furthermore, since proof trees are finite
  we have $\Gamma \vdash_{\Ax} \phi$ if and only if
  there is a finite $\Gamma' \subseteq \Gamma$ such that $\Gamma' \vdash_{\Ax} \phi$.
  Therefore $\hlcf \oplus \Ax$ is a finitary logic~\cite[Definition~1.4.1]{Kra99}.
Wherever possible, we blur the distinction between a logic and the set of its theorems (identified with derivable consecutions with an empty premise).

\subsection{The sharpening} \label{sec:sharpening}

\begin{definition}
  Let $\axDi$ be the axiom $(p \sto r) \wedge (q \sto r) \to ((p \vee q) \sto r)$,
  and define the \emph{sharp Heyting-Lewis calculus} by $\hlcs := \hlcf \oplus (\axDi)$.
\end{definition}

  The sharp systems is known to allow a simple Kripke-style semantics,
  with soundness, completeness and the finite model property results for many of its
  extensions \cite{iemh:prov01,iemh:pres03,Zhou03,IemhoffJZ05:igpl,LitVis18,grolitpat26-arxiv}:
  
\begin{definition}\label{def:sharp-model}
  A \emph{sharp frame} is a tuple $\mo{F} = (W, \leq, R)$ consisting of a set $W$,
  a partial order $\leq$ on $W$, and a relation $R$ on $W$ such that
  for all $w, v, u \in W$:
  \begin{equation*}
    \text{if}\quad w \leq v \quad\text{and}\quad v R u \quad\text{then}\quad w R u.
  \end{equation*}

  A \emph{sharp model} is formed by adding a valuation that interprets
  proposition letters as upsets.
  The interpretation of $\Lsto$-formulas at a world $w$ in a sharp model $\mo{M}$
  is defined recursively via:
  \begin{alignat*}{4}
    &\mo{M}, w \Vdashs p
      &&\iff w \in V(p) \\
    &\mo{M}, w \Vdashs \bot
      &&\phantom{\iff}\text{never} \\
    &\mo{M}, w \Vdashs \phi \wedge \psi
      &&\iff \mo{M}, w \Vdashs \phi \text{ and } \mo{M}, w \Vdashs \psi \\
    &\mo{M}, w \Vdashs \phi \vee \psi
      &&\iff \mo{M}, w \Vdashs \phi \text{ or } \mo{M}, w \Vdashs \psi \\
    &\mo{M}, w \Vdashs \phi \to \psi
      &&\iff \text{for all } w' \fgeq w,
            \text{ if } \mo{M}, w \Vdashs \phi \text{ then } \mo{M}, w \Vdashs \psi \\
    &\mo{M}, w \Vdashs \phi \sto \psi
      &&\iff \text{for all } v \, (
            \text{if } w R v
            \text{ and } \mo{M}, v \Vdashs \phi 
            \text{ then } \mo{M}, v \Vdashs \psi )
  \end{alignat*}
\end{definition}

\subsection{Algebraic semantics} \label{sec:alg}

  While Kripke completeness has so far only been available for the sharp
  calculus and its reasonably well-behaved extensions, algebra provides an
  obvious route towards a generic completeness result.

\begin{definition}
  A \emph{flat Lewisian Heyting Algebra Expansion}, or
 \emph{\hlalg}\footnote{We write $\thaes$ for the class of all Lewisian Heyting Algebra Expansions, following  Litak and Visser \cite{litvis24}. The same authors call the class of sharp algebras \textit{Lewisian Heyting Algebras with Operators} and discuss the reasons behind this terminology, whereas De Groot et al.~\cite{grolitpat26-arxiv} call the sharp algebras simply \emph{Heyting-Lewis algebras}, a name which would prove rather confusing in this context.}, 
  is a tuple $\gA \deq \la A, \wedge, \vee, \tto, \to, \bot, \top \ra$
  such that $\la A, \wedge, \vee, \to, \top, \bot \ra$ is a  Heyting algebra
  and the following laws are satisfied: 
  \begin{enumerate}\itemindent=1.1em
    \myitem{\lna{CK}}\label{it:alg-ck} $(a \tto b) \wedge (a \tto c) = a \tto (b \wedge c)$,
    \myitem{\lna{CT}}\label{it:alg-ct} $(a \tto b) \wedge (b \tto c) \leq a \tto c$,
    \myitem{\lna{CI}}\label{it:alg-ci} $a \tto a = \top$.
  \end{enumerate}
  If $\gA$ additionally satisfies
  \begin{enumerate}\itemindent=1.1em
    \myitem{\lna{CD}}\label{it:alg-cd} $(a \tto c) \wedge (b \tto c) = (a \vee b) \tto c$.
  \end{enumerate}
  then it is a \emph{sharp Lewisian Heyting Algebra} (\emph{\dhlalg}).
\end{definition}

  A \hlalg\ is a \dhlalg\ if and only if its \emph{strict reduct},
  i.e.~the reduct without $\to$, is a weak Heyting algebra~\cite{CelaniJ05:mlq}.
  We note that \lna{CK}, \lna{CD}, \lna{CT},  and \lna{CI} are referred to
  as \lna{C1}\,--\,\lna{C4} in~\cite{CelaniJ05:mlq}. 
  A valuation $v$ in $\gA$, as usual, maps propositional atoms to elements of $A$ and is 
  inductively extended to $\hat{v}$ defined on all formulas in the obvious way.
  We write $\gA, v \Vdash \phi$ if $\hat{v}(\phi) = \top$
  and $\gA \Vdash \phi$ if $\gA, v \Vdash \phi$ for every valuation $v$.
  For $\Ax \subseteq \Lsto$, we write $\hlalg(\Ax)$ for the class of
  $\hlalg$-algebras $\gA$ such that $\gA \Vdash \phi$ for all $\phi \in \Ax$.
  Furthermore, we write $\hlalg(\Ax) \Vdash \Gamma \seq \phi$ if
  there exists a finite $\Gamma' \subseteq \Gamma$ such that all
  algebras in $\hlalg(\Ax)$ validate $(\bigwedge \Gamma') \to \phi$.
  Then the usual Lindenbaum-Tarski construction gives:

\begin{theorem}
  Let $\Ax \subseteq \Lsto$ be a set of axioms and $\Gamma \seq \phi$ a consecution.
  Then $\Gamma \vdash_{\Ax} \phi$ if and only if $\hlalg(\Ax) \Vdash \Gamma \seq \phi$.
\end{theorem}

\section{Relational semantics for flat Heyting-Lewis logic}

  We introduce relational semantics for $\hlcf$, first in the most general
  version (Section~\ref{sec:flatfra}), then in the ``upward-flat'' variant
  (Section~\ref{sec:upflat}) simplifying calculations of correspondents and
  completeness proofs. In Section~\ref{sec:sharpflat} we compare the flat
  semantics to the sharp semantics from Definition~\ref{def:sharp-model}.

\subsection{Flat frames} \label{sec:flatfra}

\begin{definition}\label{def:cfrm}
  A \emph{flat frame} is a tuple $(W, \fleq, R)$ consisting of a
  nonempty set $W$, a preorder $\fleq$ on $W$ and a relation $R$ on $W$.
  A \emph{flat model} is a pair $\mo{M} = (\mo{F}, V)$ consisting of a
  flat frame $\mo{F} = (W, \fleq, R)$ and a valuation
  $V : \Prop \to \up(W, \fleq)$ that assigns to each
  proposition letter $p$ an upset $V(p)$ of $(W, \fleq)$.
  The interpretation of $\Lsto$-formulas at a world $w \in W$ extends the
  usual intuitionistic semantics with
  \begin{align*}
    \mo{M}, w \Vdash \phi \sto \psi
      &\iff \text{for all } w' \fgeq w, 
            \text{ if } \mo{M}, v \Vdash \phi 
            \text{ for all } v \in W 
            \text{ such that } w'Rv \\
      &\phantom{\iff \text{for all } w' \geq w, }\ %
            \text{ then } \mo{M}, v \Vdash \psi 
            \text{ for all } v \in W 
            \text{ such that } w'Rv
  \end{align*}
  The \emph{truth set} of $\phi$ is given by
  $\llb \phi \rrb^{\mo{M}} = \{ w \in W \mid \mo{M}, w \Vdash \phi \}$.
  
  Let $\Gamma \cup \{ \phi \} \subseteq \Lsto$
  and let $\mf{M}$ be a flat model.
  We write $\mf{M}, w \models \Gamma$ if $w$ satisfies all $\psi \in \Gamma$,
  and we say that $\mf{M}$ \emph{validates} $\Gamma \seq \phi$ if
  $\mf{M}, w \models \Gamma$ implies $\mf{M}, w \models \phi$ for all worlds
  $w$ in $\mf{M}$.
  A flat frame $\mf{F}$ \emph{validates} $\Gamma \seq \phi$ if every model
  of the form $(\mf{F}, V)$ validates the consecution, and it validates a
  formula $\phi$ if it validates the consecution $\emptyset \seq \phi$.
  If $\Ax \subseteq\Lsto$ is a set of axioms, then
  we write $\Gamma \Vdash_{\Ax} \phi$,
  and say that \emph{$\Gamma$ semantically entails $\phi$ on the class of
  flat frames for $\hlcf \oplus \Ax$}, if every flat frame that
  validates all formulas in $\Ax$ also validates the consecution $\Gamma \seq \phi$.
\end{definition}

  Using truth-set notation, we have
  $\mo{M}, w \Vdash \phi \sto \psi$ iff $R[w'] \subseteq \llb \phi \rrb^{\mo{M}}$
  implies $R[w'] \subseteq \llb \psi \rrb^{\mo{M}}$, for all $w' \fgeq w$.
  To illustrate the subtleties of this semantics (even in the classical setting),
  we give two examples showing that~\eqref{ax:di} and the reverse of~\eqref{ax:lb} are not valid.

\medskip\noindent
\begin{minipage}{.82\textwidth}
\begin{example} \label{ex:nodi}
  Consider the flat model $(W, \fleq, R)$ where $W = \{ w, v, u \}$,
  the intuitionistic accessibility relation $\fleq$ is the reflexive closure
  of the three worlds, and $R$ is given by $wRv$ and $wRu$.
  Let $V$ be a valuation such that $V(p) = \{ v \}$, $V(q) = \{ u \}$ and $V(r) = \emptyset$.
  Then $w \Vdash p \sto r$ because $w$ has no intuitionistic successor $x$ such
  that $R[x] \subseteq V(p)$, so the truth condition for $p \sto r$ is vacuously true.
  Similarly, $w \Vdash q \sto r$.
  But $w \not\Vdash (p \vee q) \sto r$, because every $R$-successor of $w$
  satisfies $p \vee q$, but not every successor satisfies $r$. This shows
  that $\axDi$ is not valid on flat frames.
\end{example}
\end{minipage}
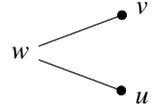
\begin{minipage}{.17\textwidth}
  \begin{flushright}
    \begin{tikzpicture}[yscale=1.2,xscale=.8]
        \node (w) at (0,0) {$w$};
        \node (v) at (2,.5) {$v$};
        \node (u) at (2,-.5) {$u$};
        \draw[-Circle] (w) to (v);
        \draw[-Circle] (w) to (u);
    \end{tikzpicture}
  \end{flushright}
\end{minipage}

\bigskip\noindent
\begin{minipage}{.82\textwidth}
\begin{example} \label{ex:classical}
  Consider the flat model depicted on the right with $V(p) = \{ s \}$ and
  $V(q) = \emptyset$.
  Then $w, v \not\Vdash \Box p$ because $u \not\Vdash p$, and hence
  $w \Vdash \Box p \to \Box q$.
  On the other hand, the fact that $R[v]$ is contained in
  $V(p)$ but not in $V(q)$ shows that $w \not\Vdash p \sto q$.
  Therefore $w \not\Vdash (\Box p \to \Box q) \to (p \sto q)$,
  i.e.~the reverse direction of~\eqref{ax:lb} is false.
\end{example}
\end{minipage}
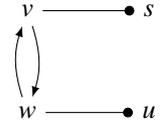
\begin{minipage}{.17\textwidth}
  \begin{flushright}
    \begin{tikzpicture}[yscale=.9,xscale=.8]
        \node (w) at (0,0) {$w$};
        \node (v) at (0,1.5) {$v$};
        \node (u) at (2,0) {$u$};
        \node (s) at (2,1.5) {$s$};
        \draw[bend left=20, -latex] (w) to (v);
        \draw[bend left=20, -latex] (v) to (w);
        \draw[-Circle] (w) to (u);
        \draw[-Circle] (v) to (s);
    \end{tikzpicture}
  \end{flushright}
\end{minipage}

\bigskip
  
  A routine induction on the structure of $\phi$ allows us to prove:

\begin{lemma}[Intuitionistic heredity]
  Let $\mo{M} = (W, \fleq, R, V)$ be a flat model.
  Then for all $\phi \in \Lsto$ and all $w, v \in W$,
  if $\mo{M}, w \Vdash \phi$ and $w \fleq v$ then $\mo{M}, v \Vdash \phi$.
\end{lemma}

  Every flat frame gives rise to a $\hlalg$ via its
  complex algebra.
  
\begin{definition}
  The \emph{complex algebra} of a flat frame $\mo{F} = (W, \fleq, R)$ is
  $\duf{\mo{F}} := \la \up(W, \fleq), \cap, \cup, \undto, \undstof, W, \emptyset \ra$,
  where
  \begin{align*}
    a \undto b &:= \{ w \in W \mid \text{for all } v \fgeq w, \, v \in a \text{ implies } v \in b \} \\
    a \undstof b &:= \{ w \in W \mid \text{for all } v \fgeq w, \, R[v] \subseteq a \text{ implies } R[v] \subseteq b \}.
  \end{align*}
\end{definition}

\begin{lemma}\label{lem:complex-alg}
  If $\mo{F} = (W, \fleq, R)$ is a flat frame then $\duf{\mo{F}}$
  is a $\hlalg$,
  and $\mo{F}$ and $\duf{\mo{F}}$ validate precisely the same consecutions.
\end{lemma}
\begin{proof}
  We know that the upsets with the given operations form a Heyting algebra.
  A routine verification shows that $\undstof$ satisfies~\eqref{it:alg-ck},
  \eqref{it:alg-ct} and~\eqref{it:alg-ci}.
  The second part of the lemma follows from the fact that valuations for
  $\mo{F}$ correspond bijectively with assignments in $\duf{\mo{F}}$,
  and that the interpretation of connectives in $\mo{F}$ corresponds to
  that in $\duf{\mo{F}}$.
\end{proof}

  Combining Lemma~\ref{lem:complex-alg} and algebraic soundness proves:

\begin{proposition}\label{prop:soundness}
  For any $\Ax \subseteq \Lsto$ and any
  $\Gamma \cup \{ \phi \} \subseteq \Lsto$, we have that
  $\Gamma \vdash_{\Ax} \phi$ implies $\Gamma \Vdash_{\Ax} \phi$.
\end{proposition}

\begin{remark} \label{rem:dualityhard}
One might expect that, at least in the finite setting, turning such a complex algebra back into a flat frame should be straightforward, with join-prime elements providing the carrier set of the frame. Example~\ref{ex:classical} illustrates that this is not the case: the Heyting reduct of the dual algebra is the Boolean algebra with three atoms (join-primes). Collapsing the $\{v,w\}$ cluster would change the equational theory. The right approach to duality, similar to the one pursued in~\cite{Wij90,groshiclo25}, uses a suitable algebraic translation of the notion  of the notion a segment introduced in Section~\ref{sec:canonical}, potentially blowing up the number of states (cf.~estimates in the proof of Lemma~\ref{lem:fmp-comp}). While many segments can often be eliminated (cf.~Remark~\ref{rem:notall} and~Section~\ref{subsec:alt-can-mod}), care is needed. 
\end{remark}

\subsection{Upward-flat frames} \label{sec:upflat}

  When only using the modal relation to interpret $\sto$, we can make a
  simplification to our frames and assume that $R[w]$ is an upset
  for each $w$.

\begin{definition}
  An \emph{upward-flat frame} is a flat frame $(W, \fleq, R)$ such that 
  for all $w, v, u \in W$, if $w R v \fleq u$ then $w R u$.
  A \emph{upward-flat model} is a flat model whose underlying frame is upward-flat.
\end{definition}

  The coherence condition on the relation can be read as $(R \circ {\fleq}) = R$.
  While not strictly required, it simplifies the correspondence
  results for some of the additional axioms we consider in Section~\ref{sec:extensions}.

\begin{proposition} \label{prop:upsame}
  Let $\mo{F} = (W, \fleq, R)$ be a flat frame.
  Define $R_{\fleq} := R \circ {\fleq}$,
  i.e.~$w R_{\fleq} u$ if there exists a $v$ such that $w R v \fleq u$,
  and let $\mo{F}_{\fleq} = (W, \fleq, R_{\fleq})$.
  Then $\mo{F}$ and $\mo{F}_{\fleq}$ have the same complex algebra.
\end{proposition}
\begin{proof}
  Let $a$ be an upset of $(W, \fleq)$ and $w \in W$. Then
  $R[w] \subseteq a$ if and only if $R_{\fleq}[w] \subseteq a$.
  This entails that the change from $\mo{F}$ to $\mo{F}_{\fleq}$
  leaves the definition of $\undstof$ unchanged,
  so that $\duf{F} = \duf{(F_{\fleq})}$.
\end{proof}

  It is often easier to find and depict frame correspondence results for upward-flat
  frames than for arbitrary ones. The definition and proposition above
  show that these can always be transformed into arbitrary frame conditions:
  simply replace every occurrence of $R$ with $(R \circ {\fleq})$. 
  To illustrate the difference,  consider the axiom $\foura: \phi \sto (\top \sto \phi)$.

\begin{proposition} \label{prop:fourcor} \
\begin{enumerate}
\item   A flat frame validates $\foura$ if and only if for all
  $x, y, z, w$ such that $x R y \fleq z R w$, there exists $v$ such that
  $x R v \fleq w$.
\item An upward-flat frame validates $\foura$ if and only if $R$ is transitive.
\end{enumerate}
\end{proposition}
\begin{proof}
  (1) \;  Suppose $\mo{W}$ validates $\foura$, and $xRy\fleq z R w$ for some worlds
  $x, y, z, w$. Let $V$ be a valuation of $p$ with $V(p) = {\uparrow}R[x]$.
  Then all worlds in $R[x]$ satisfy $p$, so by assumption they also all satisfy
  $\top \sto p$. In particular, $y \Vdash \top \sto p$,
  and since $y \fleq z$ and (trivially) all worlds in $R[z]$ satisfy $\top$,
  we must have $w \Vdash p$. By definition, this means that $w$ lies above
  some $R$-successor $v$ of $x$, as desired.
  
  Conversely, suppose $\mo{W}$ satisfies the frame condition.
  Let $x$ be any world. To show that it satisfies $\foura$,
  let $x \fleq x'$ and suppose all worlds in $R[x']$ satisfy $p$.
  Then we need that all worlds in $R[x']$ satisfy $\top \sto p$.
  Let $y$ be such a world. Since $\top$ is always true, we need to prove
  that $y \fleq z R w$ implies that $w \Vdash p$.
  This follows from the frame condition.
  So $\foura$ is valid.
  
  (2) \; This is a straightforward simplification of the above condition.
  For readers' convenience, we provide a direct proof.
  Suppose $R$ is transitive and let $w$ be a world such that $R[w] \subseteq V(p)$.
  Then by assumption $R[R[w]] \subseteq \llb p \rrb$, and since $wRv \fleq u R s$
  implies $w R u R s$ we have $R[u] \subseteq \llb p \rrb$ for every $u \in R[w]$,
  so that $R[w] \subseteq \llb \top \sto p \rrb$.
  Conversely, suppose $w R v R u$.
  Let $V$ be a valuation such that $V(p) = R[w]$.
  Then $R[w] \subseteq V(p)$, so we must have $R[w] \subseteq \llb \top \sto \phi \rrb$.
  This forces $R[R[w]] \subseteq V(p) = R[w]$.
  In particular, we have $u \in R[R[w]] \subseteq R[w]$ so $w R u$.
  Therefore $R$ is transitive.
\end{proof}

\subsection{Relation to sharp semantics} \label{sec:sharpflat}

  The sharp semantics for $\hlcs$ can be embedded into flat semantics in a
  truth preserving way. This gives rise to a completeness result for
  $\hlcs$ with respect to flat semantics. We start with a simple sufficient
  condition for a flat frame to validate $\axDi$.
  Let us call a flat frame $\mo{F} = (W, \fleq, R)$ \emph{pointwise downward
  directed} if $R[w]$ is a downward directed subset of $(W, \fleq)$
  for every $w \in W$.

\begin{lemma} \label{lem:dicom}
  If a flat frame $\mo{F} = (W, \fleq, R)$ is pointwise downward directed,
  then it validates $\axDi$
\end{lemma}
\begin{proof}
  Suppose $\mo{F}$ is a flat frame that is pointwise downward directed,
  $\mo{M} = (\mo{F}, V)$ is a flat model based on $\mo{F}$
  and $w \in W$ satisfies $p \sto r$ and $q \sto r$.
  Suppose $w' \fgeq w$ and $R[w'] \subseteq \llb p \vee q \rrb = V(p) \cup V(q)$.
  We claim that either $R[w'] \subseteq V(p)$ or $R[w'] \subseteq V(q)$.
  If this is not the case, then we can find $v, u \in R[w']$ such that
  $v \notin V(p)$ and $u \notin V(q)$. By assumption there exists some
  $s \in R[w']$ such that $s \fleq v$ and $s \fleq u$.
  But then $s \notin V(p) \cup V(q)$, a contradiction.
  So we must have $R[w'] \subseteq V(p)$ or $R[w'] \subseteq V(q)$.
  In either case, using the assumption yields $R[w'] \subseteq V(r)$.
  This proves $w \Vdash (p \vee q) \sto r$, and hence $\axDi$ is valid
  on~$\mo{F}$.
\end{proof}

  Next, we turn a sharp model into a flat one.
  Intuitively, for each $w$ we create a cluster such that each element
  of the cluster can modally access precisely one of the worlds in $R[w]$.

\begin{definition}
  Let $\mo{F} = (W, \leq, R, V)$ be a sharp frame and $\bullet$ a symbol such
  that $\bullet \notin W$.
  Define $\mo{F}^{\flat} = (W^{\flat}, \preceq, \mc{R}, V^{\flat})$, where
  \begin{align*}
    &W^{\flat} := \{ (w, v) \mid w \in W \text{ and } w R v \} \cup \{ (w, \bullet) \mid w \in W \text{ and } R[w] = \emptyset \} \quad
      & (w, v) \preceq (w', v') &\iff w \leq w' \\
    &V^{\flat}(p) := \{ (w, v) \in W^{\flat} \mid w \in V(p) \}
      & (w, v) \mc{R} (w', v') &\iff v \leq w' 
  \end{align*}
\end{definition}

  Note that the definition of $\mc{R}$ ensures that $\mo{M}^{\flat}$ is and
  upward-flat model. Moreover, for each $(w, v) \in W^{\flat}$
  we have that $\mc{R}[(w, v)]$ is the upward closure (under $\fleq$)
  of $(v, u)$, for any $u \in W \cup \{ \bullet \}$ such that $(v, u) \in W^{\flat}$.
  This implies that $\mo{M}$ is pointwise downward directed.
  
\begin{proposition}\label{prop:sharp-to-flat}
  Let $\mo{M} = (W, \leq, R, V)$ be a sharp model,
  and $\mo{M}^{\flat} = (W^{\flat}, \preceq, \mc{R}, V^{\flat})$ the
  corresponding flat model. Then for all $(w, v) \in W^{\flat}$ and
  all formulas $\phi$, we have
  $\mo{M}, w \Vdashs \phi$ if and only if $\mo{M}^{\flat}, (w, v) \Vdash \phi$.
\end{proposition}
\begin{proof}
  We use induction on the $\phi$, showcasing only the induction step for 
  $\phi = \psi \sto \chi$.
  Suppose $\mo{M}, w \Vdash^s \psi \sto \chi$.
  Suppose $(w, v) \fleq (w', v')$ and $R[(w', v')] \subseteq \llb \psi \rrb^{\mo{M}^{\flat}}$.
  Then $w \leq w'$ and $w'Rv'$, so $wRv'$ because $\mo{M}$ is a sharp model.
  Also $R[(w', v')] = \{ (u, s) \in W^{\flat} \mid v' \leq u \}$,
  so by the induction hypothesis we have $\mo{M}, v' \Vdash^s \psi$.
  The assumption that $\mo{M}, w \Vdash^s \psi \sto \chi$
  then gives $\mo{M}, v' \Vdash^s \chi$,
  and intuitionistic heredity entails $\mo{M}, u \Vdash^s \chi$ for all $u \geq v'$.
  Using induction again, this implies $R[(w', v')] \subseteq \llb \chi \rrb^{\mo{M}^{\flat}}$,
  and hence $\mo{M}^{\flat}, (w, v) \Vdash \psi \sto \chi$.
  Conversely, suppose $\mo{M}^{\flat}, (w, v) \Vdash \psi \sto \chi$
  and suppose $w R u$ and $\mo{M}, u \Vdash^s \psi$.
  Then $(w, v) \preceq (w, u)$ and by induction
  $\mc{R}[(w, u)] \subseteq \llb \psi \rrb^{\mo{M}^{\flat}}$.
  This implies $\mc{R}[(w, u)] \subseteq \llb \chi \rrb^{\mo{M}^{\flat}}$,
  and hence $\mo{M}, u \Vdash^s \chi$.
  Therefore $\mo{M}, w \Vdash^s \psi \sto \chi$.
\end{proof}

  Combining the known completeness result for $\hlcs$ with respect to sharp
  frames, the lemma and proposition above, and the fact that $\mo{M}^{\flat}$
  is a pointwise downward directed upward-flat model, gives:

\begin{proposition}
  The logic $\hlcs$ is sound and complete with respect to the class of
  pointwise downward directed flat frames.
\end{proposition}

\section{Canonical models and completeness}\label{sec:canonical}

  We provide a canonical model construction relative to some set $\Sigma$
  that is closed under subformulas. This will give us, at once, the finite
  model property and strong completeness of the logic.
  We use a modification of the canonical model construction for $\lna{CK}$,
  using so-called segments. The idea behind a segment is that it encodes
  both a world of the frame (a prime theory) as well as its successors.
  We start by defining prime $\Sigma$-theories.
  Throughout this subsection, we let $\Ax$ be a consistent set of formulas,
  and $\Sigma$ denote a set of formulas that
  contains $\top$ and is closed under subformulas.

\begin{definition}
  A \emph{prime $(\Ax,\Sigma)$-theory} is a subset $\Gamma \subseteq \Sigma$ that is
  \emph{deductively closed} (i.e.~if $\phi \in \Sigma$ and
          $\Gamma \vdash_{\Ax} \phi$ then $\phi \in \Gamma$),
  \emph{consistent} (i.e.~$\Gamma \not\vdash_{\Ax} \bot$), and
  \emph{$\Sigma$-prime} (i.e.~if $\phi_1, \ldots, \phi_n \in \Sigma$
          and $\Gamma \vdash_{\Ax} \phi_1 \vee \cdots \vee \phi_n$ then $\phi_i \in \Gamma$
          for some $i \in \{ 1, \ldots, n \}$).
  Write $\Th_{\Ax,\Sigma}$ for the set of prime $(\Ax,\Sigma)$-theories.
  If $\Sigma = \Lsto$ then we omit reference to $\Sigma$ and simply write
  \emph{prime $\Ax$-theory} instead of prime $(\Ax,\Sigma)$-theory.
\end{definition}

  The Lindenbaum lemma can be proved as usual.
  We can use it to obtain prime $(\Ax, \Sigma)$-theories by taking the intersection
  of the resulting prime $\Ax$-theory with $\Sigma$.

\begin{lemma}[Lindenbaum lemma]\label{lem:lindenbaum}
  Let $\Gamma \cup \Delta \subseteq \Lsto$ and suppose $\Gamma \not\vdash_{\Ax} \Delta$.
  Then there exists a prime theory $\Gamma'$ such that
  $\Gamma \subseteq \Gamma'$ and $\Gamma' \cap \Delta = \emptyset$.
\end{lemma}

\begin{lemma}\label{lem:sigma-pt}
  If $\Gamma$ is a prime $\Ax$-theory, then $\Gamma \cap \Sigma$
  is a prime $(\Ax,\Sigma)$-theory.
\end{lemma}

  Segments comprise of a prime $(\Ax,\Sigma)$-theory together with a suitable set
  of such theories that encodes the successors of the segment. 

\begin{definition}\label{def:can-model}
  An \emph{$(\Ax,\Sigma)$-segment} is a pair $(\Gamma, U)$ where
  $\{ \Gamma \} \cup U \subseteq \Th_{\Ax,\Sigma}$ such that
  \begin{enumerate}
    \myitem{S1} \label{it:seg-1}
          if $\Delta \in U$ and $\Delta \subseteq \Delta' \in \Th_{\Ax,\Sigma}$ 
          then $\Delta' \in U$;
    \myitem{S2} \label{it:seg-2}
          for all $\phi, \psi \in \Sigma$,
          if $\Gamma \vdash_{\Ax} \phi \sto \psi$ and $\phi \in \Delta$ for all $\Delta \in U$,
          then $\psi \in \Delta$ for all $\Delta \in U$.
  \end{enumerate}
  Let $\SEG_{\Ax,\Sigma}$ be the set of $(\Ax,\Sigma)$-segments and define relations by
  setting $(\Gamma, U) \subsetsim (\Gamma', U')$ iff $\Gamma \subseteq \Gamma'$,
  and $(\Gamma, U) \mc{R} (\Gamma', U')$ iff $\Gamma' \in U$.
  Define the (canonical) valuation by $V_{\Ax,\Sigma}(p) = \{ (\Gamma, U) \in \SEG_{\Ax,\Sigma} \mid p \in \Gamma \}$.
  Then
  \begin{equation*}
    \mo{F}_{\Ax,\Sigma} = (\SEG_{\Ax,\Sigma}, \subseteq, \mc{R})
    \quad\text{and}\quad
    \mo{M}_{\Ax,\Sigma} = (\SEG_{\Ax,\Sigma}, \subsetsim, \mc{R}, V_{\Ax,\Sigma})
  \end{equation*}
  are an upward-flat frame and model,
  called the \emph{full canonical frame} and \emph{model (with respect to $\Ax$ and $\Sigma$)}.
  If $\Sigma = \Lsto$ then we abbreviate $\SEG_{\Ax} := \SEG_{\Ax, \Lsto}$
  and $\mo{F}_{\Ax} := \mo{F}_{\Ax, \Lsto}$.
\end{definition}

  Lemmas~\ref{lem:lindenbaum} and~\ref{lem:sigma-pt} provide a way to construct
  prime $(\Ax,\Sigma)$-theories, given suitable sets of formulas.
  The following lemma allows us to extend this to a segment:

\begin{lemma}\label{lem:helper}
  Let $\Gamma$ be a prime $(\Ax,\Sigma)$-theory and $\phi \in \Sigma$, and
  define
  \begin{equation*}
    U_{\Gamma, \phi} := \{ \Delta \in \Th_{\Ax,\Sigma} \mid \text{if } \psi \in \Sigma \text{ and } \Gamma \vdash_{\Ax} \phi \sto \psi \text{ then } \psi \in \Delta \}.
  \end{equation*}
  \begin{enumerate}
    \item $(\Gamma, U_{\Gamma, \phi})$ is an $(\Ax,\Sigma)$-segment.
    \item $\phi \in \Delta$ for all $\Delta \in U_{\Gamma, \phi}$.
    \item If $\theta \in \Sigma$ is such that $\Gamma \not\vdash_{\Ax} \phi \sto \theta$,
          then there exists $\Delta \in U_{\Gamma, \phi}$ such that $\theta \notin \Delta$.
  \end{enumerate}
\end{lemma}
\begin{proof}
  (1) \;
  It follows immediately from the definition that $(\Gamma, U_{\Gamma, \phi})$
  satisfies~\eqref{it:seg-1}, so we focus on proving~\eqref{it:seg-2}.
  Suppose $\chi, \xi \in \Sigma$ and $\Gamma \vdash_{\Ax} \chi \sto \xi$
  and $\chi \in \Delta$ for all $\Delta \in U_{\Gamma, \phi}$.
  Then we must have
  \begin{equation*}
    \{ \psi \in \Sigma \mid \Gamma \vdash_{\Ax} \phi \sto \psi \} \vdash \chi,
  \end{equation*}
  because otherwise we could use the Lindenbaum lemma to find some
  prime $\Sigma$-theory in $U_{\Gamma, \phi}$ that does not contain $\chi$.
  (We can first use the usual Lindenbaum lemma to find a prime theory containing
  the LHS but not $\phi$, and then take its intersection with $\Sigma$.)
  By compactness, we can find $\psi_1, \ldots, \psi_n \in \Sigma$ such that
  $\Gamma \vdash_{\Ax} \phi \sto \psi_i$ for all $i \in \{ 1, \ldots, n \}$
  and $\psi_1, \ldots, \psi_n \vdash_{\Ax} \chi$.
  This implies
  \begin{equation*}
    \phi \sto \psi_1, \ldots, \phi \sto \psi_n \vdash_{\Ax} \phi \sto \chi,
  \end{equation*}
  hence using transitivity
  \begin{equation*}
    \phi \sto \psi_1, \ldots, \phi \sto \psi_n, \chi \sto \xi \vdash_{\Ax} \phi \sto \xi,.
  \end{equation*}
  Since $\Gamma$ derives everything on the LHS, we also get
  $\Gamma \vdash_{\Ax} \phi \sto \xi$, hence by definition of
  $U_{\Gamma, \phi}$ we have $\xi \in \Delta$ for all $\Delta \in U_{\Gamma, \phi}$.
  
  (2) \;
  This follows from the fact~($\mathsf{N_a}$) entails $\vdash \phi \sto \phi$
  for any $\phi \in \Lsto$. Therefore $\Gamma \vdash_{\Ax} \phi \sto \phi$ and hence
  $\phi \in \Delta$ for all $\Delta \in U_{\Gamma, \phi}$ by definition.
  
  (3) \;
  We claim that $\{ \psi \in \Sigma \mid \Gamma \vdash_{\Ax} \phi \sto \psi \} \not\vdash_{\Ax} \theta$.
  Suppose towards a contradiction that this is not the case.
  Then by compactness we can find $\psi_1, \ldots, \psi_n \in \Sigma$
  such that
  \begin{equation}\label{eq:3}
    \psi_1, \ldots, \psi_n \vdash_{\Ax} \theta
  \end{equation}
  and $\Gamma \vdash_{\Ax} \phi \sto \psi_i$ for each $i \in \{ 1, \ldots, n \}$.
  This implies $\Gamma \vdash_{\Ax} \phi \sto (\psi_1 \wedge \cdots \wedge \psi_n)$.
  Furthermore, \eqref{eq:3} entails $\vdash_{\Ax} (\psi_1 \wedge \cdots \wedge \psi_n) \to \theta$,
  so by~($\mathsf{N_a}$) we get $\vdash_{\Ax} (\psi_1 \wedge \cdots \wedge \psi_n) \sto \theta$.
  In particular, this gives $\Gamma \vdash_{\Ax} (\psi_1 \wedge \cdots \wedge \psi_n) \sto \theta$,
  so that $\axTr$ entails $\Gamma \vdash_{\Ax} \phi \sto \theta$, a contradiction.
  So we have $\{ \psi \in \Sigma \mid \Gamma \vdash_{\Ax} \phi \sto \psi \} \not\vdash_{\Ax} \theta$.
  Then Lemma~\ref{lem:lindenbaum} gives a prime $\Ax$-theory $\Delta$ containing
  $\psi$ for every $\psi \in \Sigma$ such that $\Gamma \vdash_{\Ax} \phi \sto \psi$,
  but not $\theta$. By definition $\Delta \cap \Sigma \in U_{\Gamma, \phi}$,
  so it is the desired witness.
\end{proof}

\begin{lemma}\label{lem:truth}
  For all $\phi \in \Sigma$ and $(\Gamma, U) \in \SEG_{\Ax,\Sigma}$ we have
  $\mo{M}_{\Ax,\Sigma}, (\Gamma, U) \Vdash \phi$ iff $\phi \in \Gamma$.
\end{lemma}
\begin{proof}
  We use induction on the structure of $\phi$.
  If $\phi$ is $\top, \bot$ or a proposition letter, the statement is immediate.
  The cases for $\wedge$ and $\vee$ follow using induction and the fact that
  $\Gamma$ is prime.
  
  \medskip\noindent
  \textit{Case $\phi = \psi \to \chi$.}
  Suppose $\psi \to \chi \in \Gamma$.
  Let $(\Gamma, U) \subsetsim (\Gamma', U')$ and suppose
  $(\Gamma', U') \Vdash \psi$. Then by definition of $\subsetsim$,
  deductive closure of prime $(\Ax,\Sigma)$-theories, and the induction hypothesis
  we find $\psi \in \Gamma'$ and $\psi \to \chi \in \Gamma'$.
  This implies $\chi \in \Gamma'$, hence by induction $(\Gamma', U') \Vdash \chi$.
  This proves $(\Gamma, U) \Vdash \psi \to \chi$.
  
  Conversely, suppose $\psi \to \chi \notin \Gamma$.
  Then $\Gamma \not\vdash_{\Ax} \psi \to \chi$, so $\Gamma, \psi \not\vdash_{\Ax} \chi$
  and we can find a prime theory $\Gamma'$ containing $\Gamma, \psi$ but not $\chi$.
  Then $\Gamma' \cap \Sigma$ is a prime $\Sigma$-theory
  and we can extend it to an $(\Ax,\Sigma)$-segment
  (for example by using Lemma~\ref{lem:helper} with $\phi = \top$)
  which (using induction) satisfies $\psi$ but not $\chi$.
  Therefore $(\Gamma, U) \not\Vdash \psi \to \chi$.
  
  \medskip\noindent
  \textit{Case $\phi = \psi \sto \chi$.}
    If $\psi \sto \chi \in \Gamma$ then we get $(\Gamma, U) \Vdash \psi \sto \chi$
    immediately from the definition of a segment.
    Now suppose $\psi \sto \chi \notin \Gamma$.
    By Lemma~\ref{lem:helper}
    $(\Gamma, U_{\Gamma, \psi})$ is a $(\Ax,\Sigma)$-segment such that
    $\psi \in \Delta$ for all $\Delta \in U_{\Gamma, \psi}$ while
    $\chi \notin \Delta$ for some $\Delta \in U_{\Gamma, \psi}$.
    Since each $\Delta$ can be extended to a segment,
    this proves $(\Gamma, U) \not\Vdash \psi \sto \chi$.
\end{proof}

  Depending on our choice of $\Ax$ and $\Sigma$, the canonical model construction
  gives rise to a finite model property and a strong completeness result.

\begin{lemma}\label{lem:fmp-comp}
  Let $\Ax$ be a set of axioms.
  \begin{enumerate}
    \item \label{it:fmp}
          Suppose that for every finite consecution $\Gamma \seq \phi$ there exists a
          finite subformula-closed set $\Sigma$ that contains $\Gamma, \phi$ and $\top$
          such that $\mo{F}_{\Ax,\Sigma}$ validates $\Ax$.
          Then $\hlcf \oplus \Ax$ has the finite model property.
    \item \label{it:comp}
          Suppose $\mo{F}_{\Ax}$ validates $\Ax$.
          Then $\hlcf \oplus \Ax$ is strongly complete with respect to the class of
          (upward-)flat frames validating $\Ax$.
  \end{enumerate}
\end{lemma}
\begin{proof}
  (1) \; 
  Let $\Gamma \cup \{ \phi \}$ be a finite set of formulas and suppose
  $\Gamma \not\vdash_{\Ax} \phi$. Let $\Sigma$ be as described.
  Then we can use Lemmas~\ref{lem:lindenbaum} and~\ref{lem:sigma-pt} to construct
  a prime $\Sigma$-theory $\Gamma'$ extending $\Gamma$ that does not contain $\phi$.
  Let $(\Gamma', U)$ be a segment in $\SEG_{\Ax,\Sigma}$.
  It follows from Lemma~\ref{lem:truth} that that
  $\mo{M}_{\Sigma}, (\Gamma', U) \Vdash \psi$ for all $\psi \in \Gamma$
  and $\mo{M}_{\Sigma}, (\Gamma', U) \not\Vdash \phi$.
  So $\mo{M}_{\Sigma} \not\Vdash \Gamma \seq \phi$,
  hence $\mo{F}_{\Ax,\Sigma}$ is a flat frame that does not validate $\Gamma \seq \phi$.

  By assumption $\Sigma$ is finite.
  A prime $\Sigma$-theory is a subset of $\Sigma$, so we have at most
  $2^{|\Sigma|}$ many prime theories, and hence at most
  $2^{|\Sigma|} \times 2^{2^{|\Sigma|}}$ $\Sigma$-segments,
  where $|\Sigma|$ denotes the size of $\Sigma$.
  Hence $\mo{F}_{\Sigma}$ is finite.
  
  (2) \;
  The proof is identical to the first paragraph of item~(1) with $\Sigma = \Lsto$.
\end{proof}

  Taking $\Ax = \emptyset$ and $\Sigma$ the closure of $\Gamma \cup \{ \phi, \top \}$
  under subformulas yields:
  
\begin{theorem}
  The logic $\hlcf$ has the finite model property and is strongly complete
  with respect to the class of (upward-)flat frames.
\end{theorem}

\begin{remark} \label{rem:notall}
  While taking the collection of all $(\Ax, \Sigma)$-segments
  in Definition~\ref{def:can-model} provides a canonical choice of segments,
  it is not strictly necessary. Analogous to~\cite{groshiclo25}, we can restrict
  the shape of segments we use while maintaining the truth lemma and completeness
  result. This can help create a canonical model that satisfies additional
  constraints.
  We will see an example of a restriction in Section~\ref{subsec:alt-can-mod},
  where we use this strategy to ensure that the modal accessibility relation
  is transitive when having $\foura$ as an axiom.
\end{remark}

\begin{remark} \label{rem:nogmt}
  A different method for obtaining completeness results, employed for
  instance for $\hlcs$~\cite{grolitpat26-arxiv} and intuitionistic modal logic
  with a $\Box$~\cite{wolterz97:al,wolterz98:lw}, is via a G{\"o}del-McKinsey-Tarski
  translation into classical bimodal logic with an S4-box $\Box_i$ and
  a normal box $\Box_m$.
  Our case seems amenable to this treatment: flat frames corresponds precisely
  to the semantics of $\lna{S4 \oplus K}$, and the interpretation of
  $\phi \sto \psi$ is given by $\Box_i(\Box_m\phi \to \Box_m\psi)$.
  However, there is a mismatch between the descriptive frames of both
  logics: a duality for $\hlcf$ would resemble that for
  $\lna{CK}$~\cite{groshiclo26} and use segments.
  As a consequence it does not seem to be the case that the two types
  of descriptive frames line up. This frustrates the transfer of e.g.~completeness.
\end{remark}

\section{Completeness and the fmp for axiomatic extensions}
\label{sec:extensions}
  
  We investigate the extension of $\hlcf$ with the axioms listed in
  Table~\ref{tab:axioms}  and the given correspondence conditions proven in
  Lemma~\ref{lem:corr-1} and Proposition~\ref{prop:fourcor}. 
  We start by using Lemma~\ref{lem:fmp-comp} to obtain
  completeness and the finite model property for certain extensions of
  $\hlcf$ with the listed axioms.
  In Section~\ref{subsec:alt-can-mod} we modify this canonical model construction
  to obtain completeness for extensions that include $\foura$, and to obtain the
  finite model property for $\hlcf \oplus \axTb \oplus \foura$.

\begin{table}[h!]
  \centering
  \begin{tabular}{lll}
    \toprule
      Axiom & Formula & Upward-flat correspondent \\ \midrule
      $\axEm$ & $p \vee \neg p$ & $\fleq$ is symmetric \\
      $\axTb$ & $(\top \sto p) \to p$ & $({\fleq} \circ R)$ is reflexive \\
      $\foura$ & $p \sto (\top \sto p)$ & $R$ is transitive \\
      $\strength$ & $(p \to q) \to (p \sto q)$ & $wRv$ implies $w \fleq v$ \\
      $\pa$ & $(p \sto q) \to (\top \sto (p \sto q))$ & if $wRvRs$ then there exists $u \fgeq w$ \\
      && such that $uRs$ and $R[u] \subseteq R[v]$ \\ \bottomrule
  \end{tabular}
  \caption{Five axioms and their correspondents for upward-flat frames $\mo{F} = (W, \fleq, R)$.}
  \label{tab:axioms}
\end{table}

\begin{remark} \label{rem:axmeaning}
Both $\foura$ and $\pa$ often occur in arithmetical contexts. It is worth noting that while $\foura$ is the ``flat'' correspondent of transitivity, $\pa$ is the ``sharp'' one \cite{litvis24}. While $\strength$ is a rather degenerate axiom classically (cf. Remark \ref{rem:degstr}), intuitionistically it plays an important role, occurring in the logics of Haskell arrows~\cite{Hug04}, guarded (co)recursion, and entailments~\cite[Section~2.4]{grolitpat26-arxiv}, and even allows a non-trivial arithmetical interpretation as \emph{completeness principle}.
\end{remark}

\begin{lemma}\label{lem:corr-1}
  Let $\mo{F} = (W, \fleq, R)$ be an upward-flat frame. Then
  \begin{enumerate}
    \item \label{it:corr-em}
          $\mo{F}$ validates $\axEm$ if and only if $\fleq$ is symmetric;
    \item \label{it:corr-tb}
          $\mo{F}$ validates $\axTb$ if and only if for all $w$ there exists $v$
          such that $w \fleq v R w$;
    \item \label{it:corr-str}
          $\mo{F}$ validates $\strength$ if and only if $w R v$ implies $w \fleq v$;
    \item \label{it:corr-pa}
          $\mo{F}$ validates $\pa$ if and only if for all $w, v, s$ satisfying
          $w R v R s$ there exists $u \fgeq w$ such that $s \in R[u]$ and $R[u] \subseteq R[v]$.
  \end{enumerate}
\end{lemma}
\begin{proof}
  \eqref{it:corr-em} \;
  Suppose $\fleq$ is symmetric.
  Let $V$ be any valuation and suppose $w \in W$ does not satisfy $p$.
  Then for all $v \fgeq w$ we have $v \fleq w$ by symmetry, so $v \not\Vdash p$.
  This proves $w \Vdash \neg p$. Therefore $\axEm$ is valid.
  For the converse, suppose the frame condition does not hold,
  so there exist $v, w$ such that $w \fleq v$ and $v \not\fleq w$.
  Let $V$ be a valuation such that $V(p) = {\uparrow}v$.
  Then $w \not\Vdash p$ because $w \notin V(p)$ and $w \not\Vdash \neg p$ because
  $w \fleq v \Vdash p$, so $\axEm$ fails.

  \eqref{it:corr-tb}
  Suppose the frame condition holds and let $V$ be any valuation.
  If $w \Vdash \top \sto p$ then for all $v \fgeq w$ we have $R[v] \subseteq p$.
  By assumption there exists such a $v$ such that $w \in R[v]$, hence $w \Vdash p$.
  Therefore $\axTb$ is valid.
  Conversely, suppose $\axTb$ is valid.
  Let $w$ be any world. Let $V$ be a valuation such that
  $V(p) = \bigcup \{ R[v] \mid v \fgeq w \}$.
  Then $w \Vdash \top \sto p$, hence $w \Vdash p$,
  so we must have $w \in R[v]$ for some $v \fgeq w$, as desired.
  
  \eqref{it:corr-str} \;
  Suppose $R \subseteq {\fleq}$, let $V$ be any valuation, and $w \Vdash p \to q$.
  If $w \fleq v$ and $R[v] \subseteq V(p)$ then by assumption $w \fleq u$ for
  all $u \in R[v]$, hence $u \Vdash q$ for all such $u$, so that $R[v] \subseteq V(q)$.
  This proves $w \Vdash p \sto q$, so $\strength$ is valid.
  For the converse, suppose the frame condition does not hold.
  Then we can find $w, v \in W$ such that $w R v$ while $w \not\fleq v$.
  Let $V$ be a valuation such that $V(p) = R[w]$ and $V(q) = {\uparrow}w$.
  (Recall that $R[w]$ is upwards closed in upward-flat frames.)
  Then $w$ trivially satisfies $p \to q$, but $w \not\Vdash p \sto q$
  because all modal successors of $w$ satisfy $p$, but not all of them
  satisfy $q$ (namely $v$ does not satisfy $q$).
  
  \eqref{it:corr-pa} \;
  Suppose the frame condition holds, and let $V$ be any valuation.
  Suppose $w \Vdash p \sto q$.
  To show that $w \Vdash \top \sto (p \sto q)$, we need to prove that
  $w \fleq w' R v$ implies $v \Vdash p \sto q$.
  To this end, let $v' \fgeq v$ and assume $R[v'] \subseteq V(p)$.
  Then because the frame is upward-flat we have $w' R v'$.
  Now let $s \in R[v']$. Then by assumption 
  there exists some $u \fgeq w'$ such that $s \in R[u] \subseteq R[v']$
  Since $w \Vdash p \sto q$ and $R[u] \subseteq V(p)$ we find $s \Vdash q$. 
  This entails that $R[v'] \subseteq V(q)$,
  so $v \Vdash p \sto q$, as desired.
 
  Conversely, if the frame condition does not hold then we can find
  $w, v, s$ such that $wRvRs$ and for all $u \fgeq w$ either $R[u] \not\subseteq R[v]$
  or $s \notin R[u]$.
  Taking $V(p) = R[v]$ and $V(q) = R[v] \setminus {\downarrow}s$ then gives
  $w \Vdash p \sto q$, because $R[u] \not\subseteq V(p)$ for all $u \fgeq w$,
  while $v \not\Vdash p \sto q$, so $w \not\Vdash \top \sto (p \sto q)$.
\end{proof}

\subsection{Reusing the full canonical model} \label{sec:reuse}

  We begin by focussing on $\axEm, \axTb, \strength$ and $\pa$.
  Towards proving completeness and the finite model property for some extensions
  of $\hlcf$ with these axioms, we give
  conditions on $\Sigma$ that guarantee that the canonical frame $\mo{F}_{\Ax, \Sigma}$
  satisfies the correspondence conditions derived in Lemma~\ref{lem:corr-1}.
  To this end, we use the following definition of single negations:
  if $\phi$ is a formula then its \emph{single negation} $\sneg\phi$ is defined
  as $\sneg\phi = \psi$ if $\phi = \neg\psi$ for some $\psi \in \Lsto$,
  and $\sneg\phi = \neg\phi$ otherwise.
  We say that a set $\Sigma$ is closed under single negations if
  $\phi \in \Sigma$ implies $\sneg\phi \in \Sigma$.

\begin{lemma}\label{lem:helper-1}
  Let $\Ax$ be a set of axioms, $\Sigma \subseteq \Lsto$ a set of formulas that
  is closed under subformulas, and
  $\mo{F}_{\Ax, \Sigma} = (\SEG_{\Ax, \Sigma}, \subsetsim, \mc{R})$
  the canonical frame generated by $\Ax$ and $\Sigma$.
  \begin{enumerate}
    \item If $\Sigma$ is closed under single negations and
          $\axEm \in \Ax$, then $\subsetsim$ is symmetric.
    \item \label{it:axTb-canon}
          If $\axTb \in \Ax$ then for all $(\Gamma, U) \in \SEG_{\Ax, \Sigma}$
          there exists $(\Delta, D)$ such that $(\Gamma, U) \subsetsim (\Delta, D) \mc{R} (\Gamma, U)$.
    \item If $\strength \in \Ax$ then $(\Gamma, U) \mc{R} (\Gamma', U')$
          implies $(\Gamma, U) \subsetsim (\Gamma', U')$
    \item If $\Sigma = \Lsto$ and $\pa \in \Ax$ then $\mo{F}_{\Ax}$ satisfies
          the correspondence condition for $\pa$.
  \end{enumerate}
\end{lemma}
\begin{proof}
  (1) \;
  Suppose $(\Gamma, U) \subsetsim (\Gamma', U')$.
  Then $\phi \in \Gamma'$ implies $\sneg\phi \notin \Gamma'$.
  Since $\Gamma \subseteq \Gamma'$ this gives $\sneg\phi \notin \Gamma$,
  hence $\phi \in \Gamma$.
  
  (2) \; 
  We can take $(\Delta, D) = (\Gamma, U_{\Gamma,\top})$.
  Then $(\Gamma, U) \subsetsim (\Gamma, U_{\Gamma,\top})$,
  and it follows from $\axTb$ that $\Gamma \in U_{\Gamma,\top}$.
  
  (3) \; 
  Suppose $(\Gamma, U) \mc{R} (\Gamma', U')$.
  Then $\phi \in \Gamma$ implies $\Gamma \vdash_{\Ax} \top \to \phi$,
  hence using $\strength$ we get $\Gamma \vdash_{\Ax} \top \sto \phi$.
  By definition of an $(\Ax, \Sigma)$-segment,
  this implies that $\phi \in \Delta$ for all $\Delta \in U$.
  It follows that $\Gamma \subseteq \Delta$ for all $\Delta \in U$.
  In particular, this implies $\Gamma \subseteq \Gamma'$,
  hence $(\Gamma, U) \subsetsim (\Gamma', U')$.
  
  (4) \;
  Suppose $(\Gamma, U) \mc{R} (\Delta, D)$.
  Then $\phi \sto \psi \in \Gamma$ implies
  $\top \sto (\phi \sto \psi) \in \Gamma$, so that $\phi \sto \psi \in \Delta$.
  It follows that $(\Gamma, D)$ is a segment.
  This implies the correspondence condition, because for any $s$ is the correspondence condition
  we can take $u = (\Gamma, D)$.
\end{proof}

\begin{theorem}
  Let $\Ax \subseteq \{ \axEm, \axTb, \strength, \pa \}$.
  Then $\hlcf \oplus \Ax$ is sound and strongly complete with respect
  to the class of (upward-)flat frames on which they are valid.
\end{theorem}
\begin{proof}
  Combine Lemma~\ref{lem:fmp-comp}\eqref{it:comp} and Lemma~\ref{lem:helper-1}.
\end{proof}

\begin{theorem}
  Let $\Ax \subseteq \{ \axEm, \axTb, \strength \}$.
  Then $\hlcf \oplus \Ax$ has the finite model property.
\end{theorem}
\begin{proof}
  Use Lemma~\ref{lem:fmp-comp}\eqref{it:fmp}, taking $\Sigma$ to be the
  closure under subformulas and under single negations of $\Gamma \cup \{ \phi \}$.
  This is finite when $\Gamma$ is finite.
  Lemma~\ref{lem:helper-1} shows that $\mo{F}_{\Ax, \Sigma}$ validates the required axiom(s).
\end{proof}

Remark \ref{rem:axmeaning} indicates that each of the axioms taken in separation and even several surprising combinations thereof (for example $\axEm \oplus \pa$) are of independent interest. In the presence of $\strength$, however, certain careless combinations may degenerate. Still, such proofs of degeneracy may also illustrate convenience of our semantics.

\begin{remark} \label{rem:degstr}
We note that $\hlcf \oplus  \axEm \oplus \strength$ is rather degenerate, reducing not only to its own $\Box$-fragment, but in fact further still to the classical propositional calculus enriched with a single constant:
One can show that $p \tto q$ is equivalent to $(p \to q) \vee \Box\bot$. While the algebraic proof is very simple, our semantics allows an even more perspicuous argument: In upward-flat frames for this system, $\fleq$ is an equivalence relation  and $R \subseteq {\fleq}$. In those clusters where $R$ is non-empty, $p \tto q$ is the same as $p \to q$, and otherwise it reduces to $\top \tto \bot$, which in such degenerate clusters is equivalent to $\top$ (and elsewhere to $\bot$).
\end{remark}

\begin{remark}
  In the logic $\hlcf \oplus \axTb \oplus \strength$ strict implication collapses to $\to$.
  To see this, note that $\strength$ already gives $(p \to q) \to (p \sto q)$.
  Combining the correspondence conditions for $\strength$ and $\axTb$
  gives: $R \subseteq {\fleq}$ and for every $w \in W$ there exists some
  $w'$ in the same $\fleq$-cluster (i.e.~$w \fleq w' \fleq w$) such that
  $R[w'] = {\uparrow}_{\fleq} w$.
  Let us verify that this entails $(p \sto q) \to (p \to q)$.
  
  Let $w$ be a world in an upward-flat model such that
  $w \Vdash p \sto q$ and let $v \fgeq w$ be a world that satisfies $p$.
  Then we can find some $v'$ in the same cluster as $v$ such that
  $R[v'] = {\uparrow}_{\fleq}v$. By assumption and intuitionistic heredity we
  then get $R[v'] \subseteq V(p)$, and since $w \fleq v'$ and $w \Vdash p \sto q$
  this implies $R[v'] \subseteq V(q)$. In particular, this gives $v \Vdash q$,
  so it follows that $w \Vdash p \to q$.
\end{remark}

\subsection{Modifying the full canonical model}\label{subsec:alt-can-mod}

  We turn our attention to extensions of $\hlcf$ with sets of axioms that
  include $\foura$. Recall that on upward-flat frames, $\foura$ corresponds
  to transitivity of the modal accessibility relation.
  The following example illustrates that we cannot use the 
  full canonical model construction from Section~\ref{sec:canonical}.

\begin{example}
  Let $\Ax = \{ \foura \}$ and consider $\Sigma = \{ \top, q \}$.
  Then we have two prime $\Sigma$-theories, $\{ \top \}$ and $\{ \top, q \}$.
  Let $\{ \Gamma \} \cup U \subseteq \{ \{ \top \}, \{ \top, q \} \}$ and suppose
  $U$ is upwards closed under inclusion.
  In order for $(\Gamma, U)$ to be an $(\Ax,\Sigma)$-segment, we need to show
  that for all $\phi, \psi \in \{ \top, q \}$,
  if $\Gamma \vdash_{\Ax} \phi \sto \psi$ and $\phi \in \Delta$ for all
  $\Delta \in U$, then $\psi \in \Delta$ for all $\Delta \in U$.
  This gives four cases, $\top \sto q$, $q \sto \top$, $q \sto q$ and $\top \sto \top$.
  The desired condition is clearly satisfied for the latter three,
  and a simple countermodel shows that $\Gamma \not\vdash_{\Ax} \top \sto q$
  for either choice of $\Gamma$.
  Therefore $(\Gamma, U)$ is an $(\Ax, \Sigma)$-segment for any choice of
  $\Gamma$ and $U$.

  In particular, this shows that for $\Gamma := \{ \top \}$ and $\Delta := \{ \top, q \}$
  we have
  $
    (\Gamma, \{ \Delta \}) \mc{R} (\Delta, \{ \Gamma, \Delta \}) \mc{R} (\Gamma, \emptyset)
  $
  while $(\Gamma, \emptyset)$ is not modally accessible from $(\Gamma, \{ \Delta \})$.
  So the modal accessibility relation $\mc{R}$ of the full canonical frame
  $\mo{F}_{\Ax,\Sigma}$ is not transitive, hence $\mo{F}_{\Ax,\Sigma}$ does not validate $\foura$.
\end{example}

  In order to prove completeness for extensions of $\hlcf$ with $\foura$,
  we used a \emph{trimmed} version of the canonical model construction
  from Section~\ref{sec:canonical}.
  This is obtained by restricting the set $\SEG_{\Ax,\Sigma}$.

\begin{definition}
  Let $\Ax$ be a consistent set of axioms and
  $\Sigma$ a set of formulas that is closed under subformulas and contains $\top$.
  We call an $(\Ax, \Sigma)$-segment $(\Gamma, U)$ \emph{pointed}
  if there exists a formula $\gamma \in \Sigma$ such that
  \begin{equation*}
    U = U_{\Gamma, \gamma} :=
        \{ \Delta \in \Th_{\Ax,\Sigma}
           \mid \text{if } \psi \in \Sigma
                \text{ and } \Gamma \vdash_{\Ax} \gamma \sto \psi
                \text{ then } \psi \in \Delta \}.
  \end{equation*}
  By Lemma~\ref{lem:helper}, every prime $(\Ax, \Sigma)$-theory can be
  extended to a pointed $(\Ax, \Sigma)$-segment.
  
  Write $\SEG^p_{\Ax,\Sigma}$ for the set of pointed $(\Ax,\Sigma)$-segments,
  and $\mo{F}_{\Ax,\Sigma}^p := (\SEG^p_{\Ax,\Sigma}, \subsetsim, \mc{R})$
  and $\mo{M}_{\Ax,\Sigma}^p := (\mo{F}_{\Ax,\Sigma}^p, V_{\Ax,\Sigma})$
  for the \emph{pointed canonical frame} and \emph{model}.
  If $\Sigma = \Lsto$ we abbreviate $\mo{F}_{\Ax}^p := \mo{F}_{\Ax, \Lsto}^p$.
\end{definition}

  Using precisely the same proof as Lemma~\ref{lem:truth}, we get

\begin{lemma}\label{lem:truth-pointed}
  Let $\Ax \subseteq \Lsto$ be a set of axioms and
  $\Sigma \subseteq \Lsto$ a set of formulas that contains $\top$ and
  is closed under subformulas. Then
  for all $\phi \in \Sigma$ and $(\Gamma, U) \in \SEG_{\Ax,\Sigma}$ we have
  $\mo{M}_{\Ax,\Sigma}, (\Gamma, U) \Vdash \phi$ iff $\phi \in \Gamma$.
\end{lemma}

\begin{theorem}
  Let $\Ax \subseteq \{ \axEm, \axTb, \strength, \foura \}$.
  Then $\hlcf \oplus \Ax$ is sound and strongly complete with respect to
  the class of (upward-)flat frames on which $\Ax$ is valid.
\end{theorem}
\begin{proof}
  It suffices to show that $\mo{F}_{\Ax}^p$ validates each of the axioms in $\Ax$.
  Using the same proof as in Lemma~\ref{lem:helper-1} shows that if
  $\axEm, \axTb$ or $\strength$ is in $\Ax$, then $\mo{F}_{\Ax}^p$ validates it,
  so we are left to consider $\foura$.
  Suppose $\foura \in \Ax$. We need to show that $\mc{R}$ is transitive.
  To this end, let
  $(\Gamma, U_{\Gamma, \gamma}) \mc{R} (\Delta, U_{\Delta,\delta}) \mc{R} (\Pi, U_{\Pi,\pi})$
  in $\mo{F}_{\Ax}^p$.
  Suppose $\gamma \sto \psi \in \Gamma$.
  By $\foura$ we also have $\psi \sto (\top \sto \psi) \in \Gamma$,
  so $\axTr$ gives $\gamma \sto (\top \sto \psi) \in \Gamma$.
  This entails $\top \sto \psi \in \Delta$, which by the definition of a segment
  gives $\psi \in \Pi$. This proves that $\Pi \in U_{\Gamma, \gamma}$,
  so that $(\Gamma, U_{\Gamma, \gamma}) \mc{R} (\Pi, U_{\Pi,\pi})$,
  as desired.
\end{proof}

  Finally, using the same kind of canonical model we derive the finite model
  property for the logic $\hlcf \oplus \axTb \oplus \foura$.
  The key insight towards this is that $\top \sto \phi$ is equivalent
  to $\top \sto (\top \sto \phi)$ in this logic,
  so that it suffices to close $\Sigma$ under ``single boxes.'' 
  
\begin{lemma}\label{lem:4a-single-box}
  We have $\vdash_{\axTb,\foura} (\top \sto \phi) \leftrightarrow (\top \sto (\top \sto \phi))$.
\end{lemma}
\begin{proof}
  As a substitution instance of $\axTb$ we get
  $\vdash_{\axTb,\foura} (\top \sto (\top \sto \phi)) \to (\top \sto \phi)$.
  Conversely, combining $\top \sto \phi$ with $\foura$ and $\axTr$
  yields $\top \sto (\top \sto \phi)$.
\end{proof}

\begin{definition}
  For $\phi \in \Lsto$ we define
  \begin{equation*}
    \boxtimes\phi
      := \begin{cases}
           \phi &\text{if } \phi = \top \sto \psi \text{ for some } \psi \in \Lsto \\
           \top \sto \phi &\text{otherwise}
         \end{cases}
  \end{equation*}
  A set $\Sigma \subseteq \Lsto$ is said to be \emph{closed under single boxes}
  if $\phi \in \Sigma$ implies $\boxtimes\phi \in \Sigma$.
\end{definition}

  Closing a finite set $\Sigma$ under single boxes at most doubles its size,
  hence it stays finite. This allows us to construct a finite model with
  a transitive modal relation.

\begin{theorem}
  The logic $\hlcf \oplus \axTb \oplus \foura$ has the finite model property.
\end{theorem}
\begin{proof}
  Let $\Gamma \seq \phi$ be a finite consecution such that
  $\Gamma \not\vdash_{\axTb, \foura} \phi$.
  Let $\Sigma$ be the set of subformulas of $\Gamma \cup \{ \top, \phi \}$
  closed under single boxes.
  Then $\Sigma$ is finite, and we can use Lemmas~\ref{lem:lindenbaum}
  and~\ref{lem:sigma-pt} to extend $\Gamma$ to a prime $(\Ax,\Sigma)$-theory
  $\Gamma'$ containing $\Gamma$ but not $\phi$.
  Lemma~\ref{lem:helper} then yields an $(\Ax, \Sigma)$-segment
  $(\Gamma', U_{\Gamma',\top})$ which by Lemma~\ref{lem:truth-pointed},
  under the canonical valuation, invalidates $\Gamma \seq \phi$.
  Therefore $\mo{F}_{\Ax, \Sigma}^p = (\SEG_{\Ax,\Sigma}^p, \subsetsim, \mc{R})$
  invalidates $\Gamma \seq \phi$.
  To establish the finite model property, we now argue that
  $\mo{F}_{\Ax,\Sigma}$ validates $\axTb$ and~$\foura$.

  Using the same proof as Lemma~\ref{lem:helper-1}\eqref{it:axTb-canon}
  shows that $\mo{F}$ validates $\axTb$.
  For $\foura$, let $(\Gamma, U_{\Gamma, \gamma})$, $(\Delta, U_{\Delta, \delta})$
  and $(\Pi, U_{\Pi, \pi})$ be three $(\Ax,\Sigma)$-segments and suppose
  $(\Gamma, U_{\Gamma, \gamma}) \mc{R} (\Delta, U_{\Delta, \delta}) \mc{R} (\Pi, U_{\Pi, \pi})$.
  Let $\gamma, \psi \in \Sigma$ and suppose $\Gamma \vdash_{\Ax} \gamma \sto \psi$.
  By assumption we have $\Gamma \vdash_{\Ax} \psi \sto (\top \sto \psi)$,
  hence by $\axTr$ we find $\Gamma \vdash_{\Ax} \gamma \sto (\top \sto \psi)$.
  This entails $\Gamma \vdash_{\Ax} \gamma \sto \boxtimes \psi$,
  and since $\psi \in \Sigma$ we have $\boxtimes\psi \in \Sigma$.
  Therefore we must have $\boxtimes\psi \in \Delta$,
  hence $\Delta \vdash_{\Ax} \top \sto \psi$.
  Finally, the definition of a segment and the fact that
  $(\Delta, U_{\Delta,\delta}) \mc{R} (\Pi, U_{\Pi, \pi})$
  entails $\psi \in \Pi$.
  Thus, we have shown that for any $\psi \in \Sigma$,
  $\Gamma \vdash_{\Ax} \gamma \sto \psi$ implies $\psi \in \Pi$,
  so that $\Pi \in U_{\Gamma,\gamma}$ hence $(\Gamma, U_{\Gamma,\gamma}) \mc{R} (\Pi, U_{\Pi,\pi})$. Therefore $\mc{R}$ is transitive, so $\mo{F}_{\Ax,\Sigma}^p \Vdash \foura$.
\end{proof}

\section{Open subframes and extension stability} \label{sec:exsta}

Litak and Visser \cite{litvis24} note a direct connection between the syntactic notion of \emph{extension stability}, motivated by arithmetical interpretations of $\tto$, and a special type of nuclei on flat algebras, more specifically \emph{open nuclei} \cite{FourmanS79,Macnab81}. Recall that nuclei provide an algebraic perspective on \emph{subframes} in modal logic  \cite{Fine85:jsl,Wolter1993,BezhanishviliG07:apal}. In particular, quotienting an algebra by an open nucleus generated by a chosen element $a$ produces an algebra (isomorphic to one) whose Heyting reduct is (isomorphic to) the ideal of elements below $a$, with suitably restricted $\tto$.  In the classical setting with a unary box, applying this construction to dual algebras of Kripke frames produces the dual algebra of the (not necessarily modally generated!) subframe induced by $a$; that is, a Kripke frame whose carrier and modal accessibility relation are restricted to $a$.
In the Heyting setting, the fact that $a$ is an element of the upset algebra means that the carrier set of the corresponding subframe is $\fleq$-generated, i.e.~an upset. When it comes to $R$, Proposition \ref{prop:upsame} indicates a certain subtlety: unlike the classical case, the dual algebras of our frames might fail to notice the presence/absence of certain $R$-edges. Let us reconsider the example of $\foura$ from Proposition~\ref{prop:fourcor}: the corresponding class of arbitrary flat frames
 does not appear closed with respect to the open subframe construction. However, over upward-flat frames, the situation changes: transitivity is well-known to be persistent with respect to subframes. Together with difficulties in presenting duality for flat subframes noted above (Remarks \ref{rem:dualityhard} and \ref{rem:nogmt}), this means that some care is needed. Given the space constraints of the present paper, we do not attempt a full discussion here. Nevertheless,  it is illustrative to provide a semantic discussion of the failure  of extension stability for $\hlcs$.

\begin{example} \label{ex:nostadi}
  Consider the flat model $(W, \fleq, R)$ where $W = \{ w, v, u, z \}$,
  the intuitionistic accessibility relation $\fleq$ is the reflexive closure
  of the four worlds together with $z \fleq v$ and $z \fleq u$, and $R$ is given by $wRv$, $wRz$ and $wRu$:
  
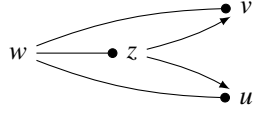
\begin{equation*}
    \begin{tikzpicture}[yscale=.6]
        \node (w) at (0,0) {$w$};
        \node (v) at (3,1) {$v$};
        \node (u) at (3,-1) {$u$};
        \node (z) at (1.5,0) {$z$};
        \draw[-Circle, bend left=15] (w) to (v);
        \draw[-Circle, bend right=15] (w) to (u);
        \draw[-Circle] (w) to (z);
        \draw[-latex, bend left=10] (z) to (u);
        \draw[-latex, bend right=10] (z) to (v);
    \end{tikzpicture}
  \end{equation*}
This frame is clearly upward-flat. Moreover, it satisfies the sufficient condition of Lemma \ref{lem:dicom} to validate $\axDi$. However, the open subframe obtained by removing $z$ is precisely the one used in Example \ref{ex:nodi} to illustrate the failure of $\axDi$. 
\end{example}

  In order to turn this counterexample into a formal proof, let us recall the
  syntactic characterisation of extension stability \cite{litvis24}.
  Given a formula $\phi$ and  a fresh propositional variable $p$, define the
  translation $\stex{p}{\phi}$ inductively as   
  commuting with the propositional variables
  and the connectives of \ipc, with the $\tto$ clause being
\begin{itemize}
  \item $\stex{p}{(\psi \tto \chi)} := ((p\to \stex{p}{\psi}) \tto (p\to \stex{p}{\chi}))$.
\end{itemize}
As $\opr\phi$ is $\top \tto \phi$, we get $\iam \vdash \stex{p}{(\opr\phi)}$ iff $\iam \vdash \opr(p\to \stex{p}{\phi})$.  Note that for any logic $\Lambda$ and any $\phi$, if $\Lambda \vdash \stex p\phi$, then $\Lambda \vdash \phi$. A logic $\Lambda$ is \emph{extension stable} if, whenever $\Lambda \vdash \phi$ and $p$ does not appear in $\phi$, we have
$\Lambda \vdash p \to \stex{p}{\phi}$.

\begin{theorem}
  The frame from Example \ref{ex:nostadi} refutes  $s \to \stex{s}{\axDi}$, i.e.
  \begin{equation*}
    s \to (((s \to p) \sto (s \to r)) \wedge ((s \to q) \sto (s \to r))
      \to ((s \to (p \vee q)) \sto (s \to r))).
  \end{equation*}
  Thus, $\hlcs$ is not extension stable, and neither is any of its extensions
  validated by this frame.
\end{theorem}

\begin{proof}
Define $V(s)$ to be the complement of $z$ and follow Example \ref{ex:nodi} for other atoms, i.e., $V(p) = \{ v \}$, $V(q) = \{ u \}$ and $V(r) = \emptyset$. One can then follow the reasoning from Example \ref{ex:nodi}, with $s$ in the antecedent used to relativize reasoning to the three-state open subframe. 
\end{proof}

For contrast, consider $\foura$. One can easily see that $\stex{s}{\foura}$ is equivalent to a substitution instance of $\foura$ itself, and hence $s \to \stex{s}{\foura}$ is a theorem of $\hlcf \oplus \foura$. This shows that the closure of the corresponding upward-flat frames under open subframes is more important than the apparent failure of such closure in the broader class. In other words, narrowing down the class of frames might be essential for giving an appropriate duality account.

\section{Conclusions and future work}

We believe we have demonstrated the potential of the flat semantics for $\hlcf$.
Future work needs to include general completeness and finite model property
results (potentially also in the context of classical subsystems of various
interpretability logics), a more systematic treatment of duality, and the open
subframe construction, possibly generalising the subframe completeness result
of Fine \cite{Fine85:jsl}.
A tantalising perspective is to use the present semantics to study combinations
of intuitionistic $\tto$ with $\Diamond$, especially on frames failing
upward-flatness.

\bibliographystyle{eptcs}
\bibliography{hlcflat}

\end{document}